\documentclass[12pt, regno]{amsart}
\usepackage{amsmath, amsthm, amscd, amsfonts, amssymb, graphicx, graphics,color}
\usepackage[breaklinks,colorlinks=true,allcolors=cyan,citecolor=red,pagebackref=false,hyperindex=false]{hyperref}
\usepackage[margin=2.9cm]{geometry}
\usepackage[mathscr]{eucal}
\usepackage{t1enc}

\newtheorem{thm}{Theorem}[section]
\newtheorem{cor}[thm]{Corollary}
\newtheorem{lemma}[thm]{Lemma}
\newtheorem{prop}[thm]{Proposition}
\theoremstyle{definition}
\newtheorem{defn}[thm]{Definition}
\theoremstyle{remark}
\newtheorem{rem}[thm]{Remark}
\numberwithin{equation}{section}
\newtheorem{ex}[thm]{Example}

\newcommand{\h}{\mathcal{H}}

\newcommand{\B}{{\mathcal B}}
\newcommand{\K}{{\mathcal K}}

\newcommand{\g}{{\mathcal G}}

\newcommand{\N}{{\mathcal N}}
\newcommand{\R}{{\mathcal R}}
\newcommand{\T}{{\mathcal T}}
\newcommand{\V}{{\mathcal V}}

\newcommand{\ltn}{{\ell}^2(\mathbb N)}

\newcommand{\mn}{\mathbb N}

\def\ep{\hfill$\square$\par\bigskip}

\def\bqs{\begin{equation}}

\def\eqs{\tag*{$\square$}\end{equation}\par\bigskip}

\def\la{\langle}

\def\ra{\rangle}

\def\ftk{\{f_k\}_{k=1}^\infty}

\def\etk{\{e_k\}_{k=1}^\infty}

\def\suk{\sum_{k=1}^\infty}

\def\bop{\begin{op}\rm}

	\def\eop{\end{op}}

\def\bee{\begin{eqnarray}}

\def\ene{\end{eqnarray}}

\def\bes{\begin{eqnarray*}}

	\def\ens{\end{eqnarray*}}

\def\bei{\begin{itemize}}

	\def\eni{\end{itemize}}

\def\bpf{\begin{pf}}

	\def\epf{\end{pf}}

\def\bt{\begin{thm}}

	\def\et{\end{thm}}

\def\bc{\begin{cor}}

	\def\ec{\end{cor}}

\def\bpr{\begin{prop}}

	\def\epr{\end{prop}}

\def\bl{\begin{lemma}}

	\def\el{\end{lemma}}

\def\bd{\begin{defn}}

	\def\ed{\end{defn}}

\def\bex{\begin{ex}}

	\def\enx{\end{ex}}

\def\bfi{\begin{fig}}

	\def\efi{\end{fig}}

\begin{document}

\title[DYNAMICAL SAMPLING:  REPRESENTATIONS AND PERTURBATIONS]{ DYNAMICAL SAMPLING: MIXED FRAME OPERATORS, REPRESENTATIONS AND PERTURBATIONS
}%

\date{}

\author{Ehsan Rashidi, Abbas Najati and Elnaz Osgooei}

	\address{\noindent Ehsan Rashidi and Abbas Najati \newline \indent Department of Mathematics
\newline
\indent Faculty of Sciences
\newline
\indent   University of Mohaghegh Ardabili
\newline \indent  Ardabil, Iran
}
\email{erashidi@uma.ac.ir,\quad a.nejati@yahoo.com,~a.najati@uma.ac.ir}
	\address{\noindent Elnaz Osgooei \newline \indent Department of Science
\newline
\indent Urmia University of Technology
\newline \indent Urmia, Iran 
}
\email{e.osgooei@uut.ac.ir}

	\begin{abstract} 
		Motivated by recent progress in operator representation of frames, we investigate the frames of the form $ \{T^n \varphi\}_{n\in I}$ for  $ I=\mathbb{N}, \mathbb{Z} $, and answer questions about representations, perturbations and frames induced by  the action of powers of bounded linear operators. As a particular case, we discuss problems concerning representation of frames in terms of iterations of the mixed frame operators. As our another contribution, we consider  frames of the form $ \{a_n T^n \varphi\}_{n=0}^{\infty} $ for some non-zero scalars $ \{a_n\}_{n=0}^{\infty} $, and we obtain some new results in dynamical sampling. Finally, we will present some auxiliary results
related to the perturbation of sequences of the form $ \{T^n \varphi\}_{n=0}^{\infty}$. 
	\end{abstract}


\maketitle

		\section{Introduction}
	A frame in a separable Hilbert space $\h$ is a countable collection of elements in $\h$ that allows each $ f\in\h$ to be written  as an (infinite) linear combination of the frame elements, but linear independence
between the frame elements is not required. Duffin and Schaeffer \cite{A88} introduced frames, and they used frames as a tool in the study sequences of the form $\{e^{i\lambda_nx}\}_{n\in\Bbb{Z}}$, where
 $\{\lambda_n\}_{n\in\Bbb{Z}}$ is a family of real or complex numbers. 
 Dynamical sampling has already introduced in \cite{A1} by Aldroubi
et al.,  and it deals with frame properties of sequences of the form $\{T^n\varphi\}_{n=0}^\infty$, where $\varphi\in\h$ and
$T :\h\to\h$ belongs to certain classes of linear operators.

 \par
 Throughout this paper,  let $\Bbb{N}_0=\{0,1,2,\cdots\}.$ We let $\h$ denote a complex separable
infinite-dimensional Hilbert space. Given a Hilbert space $\h$, we let $B(\h)$
denote the set of all bounded linear operators $T:\h\to\h$. Moreover, $GL(\h)$ will
denote the set of all bijective operators in $B(\h)$. 
	\begin{defn}
		Let $ I$ denote a countable set and let $ \{f_{k}\}_{k\in I} $ be a sequence in $ \h $.
		\bei
		\item$ \{f_{k}\}_{k\in I} $ is called a frame
		for $ \h $ if there exist constants $ A, B > 0 $ such that $ A\| f\|^{2}\leq\sum_{k\in I} | \langle f, f_{k}\rangle |^{2}\leq B\| f\|^{2}$  for all $f\in \h $; it is a frame sequence if the stated inequalities hold for all
		$ f \in \overline{\rm span}\{f_k\}_{k\in I} $ .
		\item  $ \{f_k\}_{k\in I} $ is called a Bessel sequence with Bessel bound $B$,  if $\sum_{k\in I} | \langle f, f_{k}\rangle |^{2}\leq B\| f\|^{2}$ for all $f\in \h $;
		\item  $ \{f_k\}_{k\in I} $ is called a Riesz sequence if there
		exist constants $ A, B > 0 $ such that $ A\sum_{k\in I} | c_k |^{2}\leq\| \sum_{k\in I} c_k f_k\|^{2}\leq B\sum_{k\in I} | c_k |^{2} $  for all finite scalar sequences $ \{c_k\}_{k\in I} $.
		\item   $ \{f_k\}_{k\in I} $ is called a Riesz basis for  $ \h $, if it is a Riesz sequence  for
		which $ \overline{\rm span}\{f_k\}_{k\in I}= \h $.
		\eni
	\end{defn}
	The following theorem was proved in \cite{A4} which is about frames and operators:
	\begin{thm}\label{TM1}
		Consider a sequence $ \{f_k\}_{k=1}^{\infty} $ in a separable Hilbert space $ \h $. Then the following hold:
		\bei
		\item  $ \{f_k\}_{k=1}^{\infty} $ is a Bessel sequence if and only if $ U:\{c_k\}_{k=1}^{\infty}\mapsto \sum_{k=1}^{\infty}c_k f_k $ is a well-defined mapping from $ \ell^{2}(\mathbb{N}) $ to $ \h $, i.e, the infinite series is convergent for all $ \{c_k\}_{k=1}^{\infty}\in\ell^{2}(\mathbb{N}) $; in the affirmative case the operator $ U $ is linear and bounded.\\
		\item  $ \{f_k\}_{k=1}^{\infty} $ is a frame if and only if the mapping $ \{c_k\}_{k=1}^{\infty}\mapsto \sum_{k=1}^{\infty}c_k f_k $ is well-defined from $ \ell^{2}(\mathbb{N}) $ to $ \h $ and surjective.\\
		\item  $ \{f_k\}_{k=1}^{\infty} $ is a Riesz basis if and only if the mapping $ \{c_k\}_{k=1}^{\infty}\mapsto \sum_{k=1}^{\infty}c_k f_k $ is well-defined from $ \ell^{2}(\mathbb{N}) $ to $ \h $ and bijective.
		\eni
	\end{thm}
	For $ I=\mathbb{N} $ or $ \mathbb{Z} $, \textcolor{cyan}{Theorem $ \ref{TM1}$} tells us that if $ \{f_{k}\}_{k\in I} $ is a Bessel sequence, the synthesis operator 
	\begin{equation*}
	U : \ell^{2}(I) \rightarrow \h,\quad U\{c_k\}_{k\in I} :=\sum_{k\in I} c_k f_k,
	\end{equation*}
	is well-defined and bounded. A central role will be
	played by the kernel of the operator $U$, i.e., the subset of $ \ell^{2}(I) $ given by
	\begin{equation*}
	{\N}_U = \big\{ \{c_k\}_{k\in I}\in \ell^{2}(I):~\sum_{k\in I}c_k f_k =0\big\}. 
	\end{equation*}

	The \textit{excess} of a frame is the number of elements that can be removed in order for the remaining set
to form a basis. Given a Bessel sequence $ \{f_{k}\}_{k=1}^{\infty} $, the \textit{frame operator} $ S : \h \rightarrow \h $ is defined
	by
	\begin{equation*}
	S := UU^{\ast},\quad Sf:=UU^{\ast}f=\sum_{k=1}^{\infty}\langle f, f_k\rangle f_k.
	\end{equation*} 
	\subsection{Motivation and idea of dynamical sampling}

Dynamical sampling is a recent research was introduced earlier in \cite{A1} deals with frame properties of the sequence $ \{T^n \varphi\}_{n=0}^{\infty}$ for some $ T\in(\h) $ and some $\varphi\in\h$. We will consider frames $ \{f_{k}\}_{k\in I} $ with indexing over $ I=\mathbb{N} $ or $ I=\mathbb{Z} $.  It is natural to ask whether we can find
a linear operator $ T $ such that $ f_{k+1}=Tf_k $ for all $ k\in I $. Various characterizations of frames having the form
$ \{f_{k}\}_{k\in I}=\{T^k\varphi\}_{k\in I},$ 
 where $T$ is a linear (not necessarily bounded) operator can be found in 
\cite{A7, A6, A13}. 
	We are interested in the structure of the set of iterations of the operator $T\in B(\h)$ when acting on the vector $\varphi\in\h.$  Indeed, we are interested in the following two questions:
	
	\bei 
	\item Under what conditions on $T$ and $I$ is the the iterated system of vectors $\{T^n \varphi\}_{n\in I} $ a frame or a Riesz basis for $\h$?
	\item If $\{T^n \varphi\}_{n\in I} $ is  a frame or a Riesz basis for $\h$, what can be deduced about the operator $T$? 
	\eni

	\begin{ex}
	Let $ \{e_k\}_{k=1}^{\infty} $ denote an orthonormal basis for $ \h $. Define the operator $T:\h\to\h$ by $T(f)=\sum_{k=1}^\infty \langle f, e_k\rangle e_{k+1}.$ It is clear that
	 $ \{e_k\}_{k=1}^{\infty} =\{T^k e_1\}_{k=0}^{\infty}. $
\end{ex}
\begin{ex}
	Assume that $\etk$ is an orthonormal basis for $\h,$ and
	define the bounded operator $T: \h \to \h$ by $T(f) = \sum_{k=1}^\infty  \langle f, e_k\rangle 2^{-k} e_{k+1}.$ 
	In particular, $T$ is compact, being the norm-limit of the finite-rank operators
	\begin{equation*}
	T_N: \h \to \h,\quad T_N(f) = \sum_{k=1}^N \langle f, e_k\rangle 2^{-k} e_{k+1}.
	\end{equation*}
	On the other hand, by construction the sequence $\Big\{\frac{T^{k} e_1}{\| T^{k} e_1\|}\Big\}_{k=0}^{\infty}$
	is $\etk.$
\end{ex}   
\begin{defn}\label{MixO}
	Suppose that $ \{f_k\}_{k=1}^{\infty} $ and $ \{g_k\}_{k=1}^{\infty} $ are two frames (or Bessel sequences) for $ \h $. The operator $ T:\h\to\h $ defined by $Tf=\sum_{k=1}^{\infty}\langle f, g_{k}\rangle f_k$  is called \textit{the mixed frame operator} associated with $ \{f_k\}_{k=1}^{\infty} $ and $ \{g_k\}_{k=1}^{\infty} $.
\end{defn}
Obviously, any bounded linear operator $T: \h \to \h$ is indeed a mixed frame operator. Because, if $T\in B(\h)$ and $\etk$ is an orthonormal
	basis for $\h$, then by applying $T$ on the decomposition $f=\sum_{k=1}^\infty \la f, e_k\ra e_k$, we have that
	$Tf= \sum_{k=1}^\infty \la f, e_k\ra Te_k$ for all $f\in \h.$ 
	Hence, $T$ is the mixed frame operator for the Bessel sequences $\etk$ and
	$\{Te_k\}_{k=1}^\infty$.
\par
 The following  example of a mixed frame operator was already in \cite{A13}:
\begin{ex}
	Suppose that $ \{f_k\}_{k=1}^{\infty}= \{T^n f_1\}_{n=0}^{\infty} $ is a frame for $\h$ for some  $ T\in B(\h)$. Let
 $\{g_k\}_{k=1}^{\infty}$ be a dual frame of  $ \{f_k\}_{k=1}^{\infty}$. Then
	$Tf=\sum_{k=1}^{\infty}\langle f, g_k\rangle Tf_k=\sum_{k=1}^{\infty}\langle f, g_k\rangle f_{k+1},$
	for every $ f\in\h $. Therefore,  $ T $ is  a mixed frame operator.
\end{ex}
\par
Let $ \{f_k\}_{k=1}^{\infty}$ be a Bessel sequence and $\etk$ be an orthonormal basis for $\h$. Define the operator $T:\h \to\h$ by 
$Tf=\sum_{k=1}^{\infty}\langle f, e_k\rangle f_k$. It is clear that $T$ is bounded and $Te_k=f_k$ for all $k$. Therefore we have the following: 
\begin{prop}
 The Bessel sequences in $\h$ are precisely the sequences $\{Te_k\}_{k=1}^\infty,$ where $T\in B(\h)$ and $\etk$ is an orthonormal basis for $\h$. 
\end{prop}
	\subsection{Recent results on dynamical sampling and frames}

Various aspect of the dynamical sampling problem and related frame theory have been studied by Aldroubi et al. and Christensen et al. in \cite{A1, A2, A3, A13, A113, A7, A6, A116, A117}. They deal with frame properties of sequences in a Hilbert space $ \h $ of the form $ \{T^n \varphi\}_{n=0}^{\infty} $, where $ \varphi\in\h $ and $ T\in B(\h) $. However, some no-go results in dynamical sampling have been proved; for example, if $ T $ is a normal operator, then $ \{T^n \varphi\}_{n=0}^{\infty} $ cannot be a basis \cite{A2}. Moreover, if $ T $ is a unitary operator or a compact operator,  then $ \{T^n \varphi\}_{n=0}^{\infty} $ cannot be a frame \cite{A3, A13}.
The following recent results in dynamical sampling and frame representations with bounded operators can be found in  \cite{A13, A7, A6, A117}. 
Suppose that $ \{f_k\}_{k=1}^{\infty} $ is a frame for $ \h $: 
\bei  
\item[$(i)$]  $ \{f_k\}_{k=1}^{\infty}$ has a representation$ \{f_k\}_{k=1}^{\infty}=\{T^k f_1\}_{k=0}^{\infty} $ for some bounded operator $T: \h \to\h$ if and only if $ \{f_k\}_{k=1}^{\infty} $ is linearly independent.
\item [$(ii)$]  Let $T: {\rm span}\{f_{k}\}_{k=0}^{\infty}\to {\rm span}\{f_{k}\}_{k=0}^{\infty}$ be a linear operator and $ \{f_k\}_{k=1}^{\infty}=\{T^k f_1\}_{k=0}^{\infty} $. Then $T$ is bounded if and only if the kernel $ \N_{U} $ of the synthesis operator is invariant under right-shifts; in particular $ T $ is bounded if $ \{f_k\}_{k=1}^{\infty}=\{T^k f_1\}_{k=0}^{\infty} $ is a Riesz basis.
\item [$(iii)$]  Assume that $ \{f_k\}_{k=1}^{\infty} $ is linearly independent and overcomplete. Then $ \{f_k\}_{k=1}^{\infty} $ has infinite excess.
\eni 
\par
For countable subsets $ \g\subset \h $ and a normal operator $ T $, Aldroubi et al. \cite{A2} proved that the iterative system $ \{T^n \varphi\}_{\varphi\in \g, n\geq 0} $ can be a frame for $ \h $, but cannot be a basis. However, it is difficult for a system of vectors of the form $ \{T^n \varphi\}_{\varphi\in \g, n\geq 0} $
to be a frame. The difficulty is that the  the spectrum of $ T $ must be very special. Such frames however do exist, as shown by
the constructions in \cite{A1}.

\par
The paper is organized as follows. In section $ 2 $, we provide an alternative proof to show that $\bigcup_{j=1}^{k}\{T^n \varphi_j\}_{n=0}^{\infty} $ cannot form a frame for $ \h $, whenever $ T $ is compact. Moreover, we provide necessary and sufficient conditions for $T$ being surjective. The main purpose of this section is to characterize and compare the Bessel and frame properties of orbits $ \{T^n \varphi\}_{n=0}^{\infty} $ with a bounded operator $ T $ in connection with frame operators and mixed frame operators.  We also show that the iterative actions of the mixed frame operator associated with two orthonormal basis cannot form a frame. 
Section $ 3 $ discusses representations of frames which can be represented of the form $ \{a_n T^n \varphi\}_{n=0}^{\infty} $ for some non-zero scalars $ \{a_n\}_{n=0}^{\infty} $ with $ \sup_n \big|\dfrac{a_n}{a_{n+1}}\big|<\infty$. 
Finally, in section 4 we illustrate some auxiliary results related to the perturbation of an operator to construct frame orbits in terms of the operator representations.

\section{Iterative actions of frame operator and mixed frame operator}

The representation of frames in the form $ \{T^n \varphi\}_{n=0}^{\infty} $ and $ \{T^n \varphi\}_{n\in\mathbb{Z}} $ for some $ \varphi\in\h $ and some  $T\in B(\h) $ was already studied in \cite{A13, A7}. Aldroubi et al.  \cite{A1} showed that iterative actions of compact self-adjoint operators cannot form a frame. However, for a normal operator, Philipp  \cite{A11} proved that $ \{T^n \varphi\}_{n\in \mathbb{N}} $ can be a Bessel sequence.
It is clear that the iterative system $ \{ T^{n} \varphi\}_{n=0}^{\infty} $is a Bessel sequence if $ \|T\|<1 $. Indeed, for any $ f\in \h,$ we have
\begin{equation*}
\begin{aligned}
\sum_{n=0}^{\infty}\vert \langle f, T^{n}\varphi \rangle\vert^{2}\leq\sum_{n=0}^{\infty}\|f\|^{2}\| T^{n}\varphi\|^{2}\leq\|f\|^{2}\|\varphi\|^{2}\sum_{n=0}^{\infty}\| T\|^{2n}
= \dfrac{\|\varphi\|^{2}}{1-\|T\|^{2}}\|f\|^{2}.
\end{aligned}
\end{equation*}
\par
It has already proved that if $ T $ is a compact operator on an infinite-dimensional Hilbert space $ \h $ and $ \varphi_1, ..., \varphi_k\in\h $, then $ \bigcup_{j=1}^{k}\{T^n \varphi_j\}_{n=0}^{\infty} $ cannot be a frame for $ \h $ \cite{A13}. Here we provide an alternative simple proof. We first prove a lemma.
\begin{lemma}\label{lemRT}
Let $T\in B(\h)$ and $ \varphi_1, ..., \varphi_k\in\h $.  If $ \bigcup_{j=1}^{k}\{T^n \varphi_j\}_{n=0}^{\infty} $ is a frame for $\h$, then $T$ has closed rang and the range of $T$ is $\R_T=\overline{\rm span}\{T^n \varphi_j:  j=1,2,\cdots,k \}_{n=1}^{\infty}.$ 
\end{lemma}
\begin{proof}
For each $ x\in\h $ there exists a sequence $ \{c_{n,j}:  j=1,2,\cdots,k \}_{n=0}^{\infty} $ of scalars such that $x=\sum_{j=1}^{k}\sum_{n=0}^{\infty}c_{n,j}  T^n \varphi_j $. Therefore
	\begin{equation*}
	Tx=\sum_{j=1}^{k}\sum_{n=0}^{\infty}c_{n,j}  T^{n+1} \varphi_j\in\overline{\rm span}\{T^n \varphi_j:  j=1,2,\cdots,k \}_{n=1}^{\infty}.
	\end{equation*}
	Therefore $ \R_T \subseteq {\K}:=\overline{\rm span}\{T^n \varphi_j:  j=1,2,\cdots,k \}_{n=1}^{\infty}$. On the other hand, since $ \bigcup_{j=1}^{k}\{T^n \varphi_j\}_{n=1}^{\infty} $ is a frame for $\K $,  for each $ x\in K $ 
 there is a sequence $ \{c_{n,j}:  j=1,2,\cdots,k \}_{n=1}^{\infty} $ of scalars such that $ x=\sum_{j=1}^{k}\sum_{n=1}^{\infty}c_{n,j}  T^n \varphi_j =T\big(\sum_{j=1}^{k}\sum_{n=0}^{\infty}c_{n,j}  T^n \varphi_j\big) \in \R_T$. Therefore  $\R_T=\overline{\rm span}\{T^n \varphi_j:  j=1,2,\cdots,k \}_{n=1}^{\infty}$, i.e., $ T $ has closed range.
\end{proof}
\begin{prop}
 Suppose that $ \dim \h=\infty $, $ \varphi_1,\cdots,\varphi_k\in\h $ and $ T:\h\rightarrow\h $ is a compact operator. Then $\bigcup_{j=1}^{k}\{T^n \varphi_j\}_{n=0}^{\infty} $ cannot form a frame for $ \h $.
\end{prop}
\begin{proof}
 Let $\bigcup_{j=1}^{k}\{T^n \varphi_j\}_{n=0}^{\infty} $ be a frame for $ \h$. Then $T$ has closed rang and $\R_T=\overline{\rm span}\{T^n \varphi_j:  j=1,2,\cdots,k \}_{n=1}^{\infty}$  by  \textcolor{cyan}{Lemma \ref{lemRT}}.  We denote by $T^{\dagger}\in B(\h)$  the pseudo-inverse of $T$, i.e.,  
\[T^{\dagger}: \h\to \h,\quad TT^{\dagger}x=x,\quad x\in \R_T.\]
Since $ T $ is compact, $ TT^{\dagger}=I_{\R_T} $ is compact. This implies that    $\R_T $ is finite-dimensional, and it leads to conclude $ \dim \h< \infty $, which is a contradiction. Therefore $\bigcup_{j=1}^{k}\{T^n \varphi_j\}_{n=0}^{\infty} $ cannot be a frame for $ \h $.
 \end{proof}
As we saw in  \textcolor{cyan}{Lemma \ref{lemRT}}, $\R_T$ is closed if $\{T^n\varphi\}_{n=0}^\infty$ is a frame. The following proposition provides  necessary and sufficient conditions for $T$ being surjective.
\begin{prop}
Let $T\in B(\h)$ and $\varphi\in\h$. Assume that $\{T^n\varphi\}_{n=0}^\infty$ is a frame for $\h$ with frame operator  $S$. Then the following hold:
\bei
\item[$(i)$] $T$ is surjective if and only if there exists $n\geq1$ such that $\langle T^n\varphi, S^{-1}\varphi\rangle\neq0.$
\item[$(ii)$] $T$ is surjective if and only if $\varphi\in \R_T.$
\item[$(iii)$] $T$ is surjective if and only if $S^{-1}\varphi\notin\ker{T^{\ast}}.$
\item[$(iv)$] $T$ is surjective if and only if $\|S^{-{1/2}}\varphi\|\neq1.$
\eni
\end{prop}
\begin{proof}
$(i)$ First assume that $T$ is surjective. Then $\h=\overline{\rm span}\{T^n \varphi\}_{n=1}^{\infty}$ by \textcolor{cyan}{Lemma \ref{lemRT}}. If $\langle T^n\varphi, S^{-1}\varphi\rangle=0$ for all $n\geq1$, then $ S^{-1}\varphi\perp\h$. This implies that $\varphi=0$, which is a contradiction. 
Conversely, assume that $\langle T^n\varphi, S^{-1}\varphi\rangle\neq0$ for some $n\geq1$. Then 
\[T^n\varphi=\sum_{i=0}^\infty\langle S^{-1}T^n\varphi, T^i\varphi\rangle T^i\varphi=
\langle T^n\varphi, S^{-1}\varphi\rangle\varphi+\sum_{i=1}^\infty\langle S^{-1}T^n\varphi, T^i\varphi\rangle T^i\varphi.\]
Therefore $\varphi\in \R_T$. On the other hand, $\{T^n\varphi\}_{n=1}^\infty$ is a frame sequence, and  
$\R_T=\overline{\rm span}\{T^n \varphi\}_{n=1}^{\infty}$ by \textcolor{cyan}{Lemma \ref{lemRT}}. Hence  $\varphi\in \R_T$ implies that $\R_T=\overline{\rm span}\{T^n \varphi\}_{n=1}^{\infty}=\overline{\rm span}\{T^n \varphi\}_{n=0}^{\infty}=\h,$ as desired. The result in $(ii)$ follows from the proof of $(i)$. To prove $(iii)$, it follows from $(i)$ that $T$ is surjective if and only if  $S^{-1}\varphi\notin[\R_T]^\perp =\ker{T^{\ast}}.$ For the proof of $(iv)$, assume
that $T$ is surjective and $\|S^{-{1/2}}\varphi\|=1.$ Since 
\begin{equation}\label{iv}
\varphi=\langle S^{-1}\varphi,\varphi\rangle\varphi+
\sum_{n=1}^\infty\langle S^{-1}\varphi,T^n\varphi\rangle T^n\varphi,
\end{equation}
we get $\sum_{n=1}^\infty\langle S^{-1}\varphi,T^n\varphi\rangle T^n\varphi=0.$ Then $\sum_{n=1}^\infty\big|\langle S^{-1}\varphi,T^n\varphi\rangle \big|^2=0.$ Applying $(i)$, we conclude that $T$ is not surjective, which is a contradiction. Conversely,  if $\|S^{-{1/2}}\varphi\|\neq1$, then  (\ref{iv}) implies that  there exists $n\geq1$ such that $\langle T^n\varphi, S^{-1}\varphi\rangle\neq0.$ Hence $T$ is surjective by $(i)$.
\end{proof}
Since a Riesz base and its canonical dual are bi-orthogonal, we have
\begin{cor}
Let $T\in B(\h)$ and $\varphi\in\h$. Assume that $\{T^n\varphi\}_{n=0}^\infty$ is a Riesz basis for $\h.$ Then $T$ is not surjective. In particular, $\varphi\notin \R_T$ and $S^{-1}\varphi\in\ker T^{\ast}.$
\end{cor}
Let $ \{f_{k}\}_{k=1}^{\infty} $ be a  frame for $\h$ with frame operator $S$. We investigate the   question:
 Does there exist some $ \varphi\in {\h} $ such that $\{S^{n}\varphi\}_{n= 0}^ {\infty}$ is a frame?
There are many frames for which this cannot happen. For example, if $ \{f_{k}\}_{k=1}^{\infty} $ is a tight frame for $\h$ with bound $ A $, then for $\varphi (\neq0)\in\h$, we have
\begin{equation*}
\sum_{n=0}^{\infty}| \langle f, S^{n}\varphi\rangle |^{2}=\sum_{n=0}^{\infty}| \langle f, A^{n}\varphi\rangle |^{2}=| \langle f, \varphi\rangle |^{2}\sum_{n=0}^{\infty} A^{2n},\quad f\in\h.
\end{equation*}
Therefore, $ \{S^{n}\varphi\}_{n= 0}^ {\infty} $ is a frame for $\h$ if and only if $\dim\h=1$ and $A<1$.
\par
The following exhibits a concrete example of a frame $\ftk= \{T^n f_1\}_{n=0}^\infty$
for which $T$ is a frame operator:
\bex\label{6057a} Consider the operator $T:\ltn\to\ltn$ defined by
\begin{equation}\label{6057b}
T\{c_k\}_{k=1}^\infty=\{(1-2^{-k})c_k\}_{k=1}^\infty,\quad \{c_k\}_{k=1}^\infty\in\ltn.
\end{equation}
Letting $\lambda_k=1-2^{-k}$ for $k\in\mn,$
Aldroubi et al. \cite{A1} proved that the sequence $\{ T^n b\}_{n=0}^\infty$ is a frame for $\ltn$
whenever  $b=\{\sqrt{1-\lambda_k^2}\}_{k=1}^\infty$.
Defining the bounded operator $U:\ltn\to\ltn$ by $U\{c_k\}_{k=1}^\infty=\{\sqrt{1-2^{-k}}c_k\}_{k=1}^\infty$,
we have $U=U^{\ast}$ and $T=U^2.$  Let $\{\delta_k\}_{k=1}^\infty$ be the standard basis of $\ltn$ and let $S$ be the frame operator of  $ \{U\delta_k \}_{k=1}^\infty =\{ \sqrt{1-2^{-k}}\delta_k  \}_{k=1}^\infty$.  Then
\bes Sf=  \sum_{k=1}^\infty \la f, U\delta_k\ra U\delta_k=
U \sum_{k=1}^\infty \la U^*f, \delta_k\ra \delta_k=UU^*f=Tf,\quad f\in\ltn,
\ens
i.e., $S=T.$
\ep \enx
Motivated by \textcolor{cyan}{Example $ \ref{6057a} $}, we can characterize the case that a frame  has a representation $\{T^n \varphi\}_{n=0}^{\infty} $, where $ T $ is a frame operator. Indeed, we show that positive and invertible operators are a characteristic of  frame operators: 
\begin{prop}
	Let $ T\in B(\h)$. Then the followings are equivalent:
\bei
	\item [$ (i)$] $ T $ is positive and invertible.
	\item [$ (ii) $] $ T $ is the frame operator for a frame.
	\eni
\end{prop}
\begin{proof}
	To prove $ (i)\Rightarrow (ii) $, consider the bounded and surjective operator $ U : \h \rightarrow \h $ such that $ T= UU^{\ast} $. Let $ \{e_{k}\}_{k=1}^{\infty} $ denote an orthonormal basis for $ \h $, and let $ f_{k}=Ue_{k} $ for each $k\in\Bbb{N}$. Then $ \{f_{k}\}_{k=1}^{\infty} $ is a frame and its frame operator $ T $ because
	\begin{equation*}
	Tf=UU^{\ast}f=\sum_{k=1}^{\infty}\langle  f, Ue_{k}\rangle Ue_{k}=\sum_{k=1}^{\infty}\langle  f, f_{k}\rangle f_{k}, \quad  f\in \h.
	\end{equation*}
	This proves $ (ii) $. The implication $ (ii)\Rightarrow (i) $ is clear.
\end{proof}
In the following proposition we provide a necessary condition for $ \{S^{n}g\}_{n\geq 0, g\in \g} $ to be a frame, where $ \g\subset \h $ is a countable set.
\begin{prop}
	Assume that $ \{f_{k}\}_{k=1}^{\infty} $ is a frame with lower frame bound $ A $ and frame operator $ S $. If $\g$ is a countable subset of  $\h $, and  $\{S^{n}g\}_{n\geq 0, g\in \g} $ is a frame for $ \h $, then $A<1 $.
\end{prop}
\begin{proof}
	Since $ A \langle f, f \rangle \leq \langle Sf, f \rangle $, we get
	$A\|f\|\leq\| Sf \|$ for all $f\in\h$. Therefore,
	\begin{equation*}
	\langle S^{2}f, f\rangle = \langle Sf, Sf \rangle = \| Sf \|^{2}\geq A^{2}\|f\|^{2}=A^{2}\langle f, f \rangle,
	\end{equation*}
	and then $ A^2\|f\|\leq \|S^2 f\| $  for all $f\in\h$. 
	By Induction, we conclude that for each positve integer $ m,$
	\begin{equation*}
	\begin{aligned}
	A^{m}\| f\|\leq \| S^{m}f \|, \quad f\in\h.
	\end{aligned}
	\end{equation*}
	Since $ \{S^{n}g\}_{n\geq 0, g\in \g} $ is a frame for $ {\h} $,  we get $\| S^{m}f\| \rightarrow 0$ as $m\rightarrow \infty$ for all $f\in\h$ by [\cite{A3}, \textcolor{cyan}{Theorem 7 }]. 
	Then $A^{m}\rightarrow 0$ as $m\rightarrow \infty,$ and this leads to get $ A< 1 $.
\end{proof}
\begin{rem}
	Suppose that $ \{f_k\}_{k=1}^{\infty} $ is a frame for $ \h $ with lower bound $ A $. Let $ S $ be the frame operator for $ \{f_k\}_{k=1}^{\infty} $ such that $ V\subset\h $ is an invariant subspace under $ S $. If there exists $ \lambda\in [0, 1) $ such that $ \| S\varphi\|\leq \lambda\| \varphi\|$ for all $\varphi\in V $,
	then $ \{S^n \varphi\}_{n=0}^{\infty} $ is a Bessel sequence for all $ \varphi\in V $. Indeed, for all $ f\in \h $ and $ \varphi\in V $, we have that
	\begin{equation*}
	\sum_{n=0}^{\infty}|\langle f, S^n \varphi\rangle |^{2}\leq \| f\|^{2}\sum_{n=0}^{\infty}\| S^n \varphi\|^2\leq \| f\|^2 \sum_{n=0}^{\infty} \lambda^{2n}=\dfrac{\| f\|^2}{1- \lambda^{2}}.
	\end{equation*}
\end{rem}
It follows from [\cite{A3}, \textcolor{cyan}{Theorem 7 }] that
for any unitary operator $ T: {\h}\rightarrow {\h} $ and any set of vectors $ G \subseteq {\h} $, $ \{T^{n}g\}_{g\in \g, n\geq 0} $ is not a frame .
\begin{prop}\label{Gl}
	Let $\{e_{k}\}_{k=1}^{\infty}$ and $\{\delta_{k}\}_{k=1}^{\infty}$ denote two orthonormal bases for a Hilbert space $\h$, and consider
	the mixed frame operator
	\bes T: \h \rightarrow \h, \quad Tf= \suk \la f, e_k \ra \delta_k.   \ens
	Then $\{T^n \varphi\}_{n=0}^\infty $ cannot be a frame for $\h$ for any $\varphi \in \h.$ 
\end{prop}
\begin{proof} Since $Te_j= \delta_j$ for all $j\in \mn,$ the operator $ T $ maps the orthonormal basis
	$\{e_{k}\}_{k=1}^{\infty}$ onto the orthonormal basis $\{\delta_{k}\}_{k=1}^{\infty}$. Therefore $ T $ is unitary.  By [\cite{A3}, \textcolor{cyan}{Corollary 2}], we conclude that  $\{T^n \varphi\}_{n=0}^\infty $ is not  a frame for $\h$ for any $\varphi \in \h.$ 
\end{proof}
By use of \textcolor{cyan}{Theorem $ \ref{TM1} $} we get some useful results related to iterative actions of a mixed frame operator:
\begin{cor}
\bei
	  Suppose that $ \{e_{k}\}_{k=1}^{\infty} $ and $ \{\delta_{k}\}_{k=1}^{\infty} $ are orthonormal bases for $ \h $. The following statements hold:\\	
   \item[$ (i)$] Let $\{Ue_{k}\}_{k=1}^{\infty} $ be  a Riesz basis  for $ \h $  and $Gf:= \suk \la f, \delta_k \ra Ue_k$ for all $f\in\h$, where  $ U\in GL(\h) $ is a bounded bijective  operator.
     If $ \{G^{n}\varphi\}_{n=0}^{\infty} $  is a frame for some $ \varphi\in \h $, then $ \|U\|\geq 1 $.\\
	\item[$ (ii) $] Let $\{Ue_{k}\}_{k=1}^{\infty} $ and $\{V\delta_{k}\}_{k=1}^{\infty} $ be two frames for $ {\h} $ and $Gf:= \suk \la f, V\delta_k \ra Ue_k$ for all $f\in\h$,  where  $ U, V:\h\to\h $ are bounded surjective linear operators.
 If $ \{G^{n}\varphi\}_{n=0}^{\infty} $ is a frame for $ \h $, then $ \| U \|\|V\|\geq 1 $.
\eni
\end{cor}
\begin{proof}
$ (i) $	We define the operator $T:\h\to \h$ by $ Tf=\sum \langle f, \delta_{k} \rangle e_{k} $. It is clear that $T$ is isometric, and $Gf=UTf$ for all $f\in\h$. Therefore, $\|G\|\leq\|U\|$.
 On the other hand, [\cite{A3}, \textcolor{cyan}{Theorem 9}] shows that $ \| G \|\geq 1 $, which yields the result.\\
$ (ii) $ Let $T$ as in $(i)$. Therefore $G=UTV^{\ast}$, and we get $\|G\|\leq\|U\|\|V\|$. Hence, $\|U\|\|V\|\geq1$ by [\cite{A3}, \textcolor{cyan}{Theorem 9}].
\end{proof}
    \begin{cor}
   Suppose that $\{e_{k}\}_{k=1}^{\infty}$ and $\{\delta_{k}\}_{k=1}^{\infty}$ are two orthonormal bases for a Hilbert space $\h$.
    	\bei
    		\item [$ (i) $]  Let $ \ftk $ be a Parseval frame for $ \h $ and let $T$ be the mixed frame operator defined by $Tf=\sum_{k=1}^\infty\langle f,f_k\rangle e_k$.  If $ \{T^{n}\varphi\}_{n=0}^{\infty} $ is a frame for $ \h $ for some $\varphi\in\h$, then $ T $ is not a surjective operator.\\
    		\item[$ (ii) $] Let $\{ U\delta_{k}\}_{k=1}^{\infty} $ be a frame for $ {\h} $ and $Tf=\sum_{k=1}^{\infty} \langle f, U\delta_{k}\rangle e_{k}$, where  $ U:\h\rightarrow\h $ is a bounded surjective linear operator. If $ \{T^{n}\varphi\}_{n=0}^{\infty} $ is a frame for $ {\h} $ for some $\varphi\in\h$, then $ U^{\ast}U\neq I $, i.e., $U$ is not isometric.
    	\eni
	\end{cor}
\begin{proof}
		$ (i) $ 
	Since $ \ftk $ is a Parseval frame, we have $ \| Tf\|^{2}=\sum_{k=1}^{\infty} \vert\langle f, f_{k}\rangle\vert ^{2} = \| f\|^{2}$ for all $f\in\h$.
	 Then $ T^{\ast}T=I $. If we suppose that $ T $ is surjective, then $ T $ is unitary. Using [\cite{A3}, \textcolor{cyan}{Corollary 2}], we conclude that $ \{T^{n}\varphi\}_{n=0}^{\infty} $ is not a frame for $ {\h} $. For part $ (ii) $, if $ U^{\ast}U=I $ and $ U $ is surjective, then $ U $ will be a unitary operator. Since $ TU \delta_{k}=e_k$ for all $k\in\mathbb{N}$, we get  $ TU $ is unitary. Therefore $T$ is unitary, and then $ \{T^{n}\varphi\}_{n=0}^{\infty} $ cannot be a frame for $ {\h} $.
\end{proof}
In the case of normal operators, we have the following result for infinite dimensional Hilbert spaces:
\begin{lemma}\label{NOR}
Suppose that $ T:\h\rightarrow\h $ is a normal operator and $ \varphi\in\h $ such that $ \{T^{n}\varphi\}_{n=0}^{\infty} $ is a frame for $ \h $. Then $ \| T \|= 1 $.
\end{lemma}
\begin{proof}
 Using [\cite{A2}, \textcolor{cyan}{Theorem 5.7}],  we have $T=\sum_{j=0}^{\infty}\lambda_j P_j$,
where each
$P_j$ is a rank one orthogonal projection such that $\sum_{j} P_j=I$ , $P_j P_i=0$ for all $ j\neq i$,  and $ |\lambda_j|< 1$ 
for all $j\in\mathbb{N}$.
Since $ \sum_{j} P_j=I $, we have that $\| f\|^2=\sum_{j}\|P_{j}f\|^2 $ for all $f\in\h$. Therefore 
\begin{equation*}
\begin{aligned}
\|Tf\|^2=\sum_{j}|\lambda_j|^2\|P_j f\|^2
&\leq\sum_{j}\|P_j f\|^2=\|f\|^2,\quad f\in\h.
\end{aligned}
\end{equation*}
Therefore $ \| T\|\leq 1 $. On the other hand, we have $ \| T\|\geq 1 $ by [\cite{A3}, \textcolor{cyan}{Theorem 9}], which leads to the desired result.
\end{proof}
\begin{prop}
	Let $T\in B(\h) $  and $ \varphi\in \h $ be such that $ \{T^{n}\varphi\}_{n=0}^{\infty} $ is a frame for $ \h $. 
\bei
	\item[$(i)$] There exists a countable set $\g\subset\h $ such that $ \{V^{n}\psi\}_{\psi\in \g, n\geq 0} $ is a tight frame for $\h$, where $V= \|T\|^{-1}T$.
	\item[$ (ii)$]  If $ T $ is a normal operator, then there exists a countable set $\g\subset\h $ such that $ \{(TT^{\ast})^{n}\psi\}_{\psi\in \g, n\geq 0} $ is a tight frame for $\h$.
\eni
\end{prop}
\begin{proof}
	$ (i) $ By using of [\cite{A3}, \textcolor{cyan}{Theorems 7, 9}],  we have $ \|T\|\geq 1 $ and $ (T^{\ast})^{n}f \rightarrow 0 $ for all $ f\in \h $ as $ n\rightarrow\infty $. Since $ \|V\|=1$  and  $ (V^{\ast})^n f\rightarrow 0$ for all $ f\in \h $ as $ n\rightarrow\infty $, the result follows from [\cite{A3}, \textcolor{cyan}{Theorem 8}]. In order to prove $ (ii) $, since $ \{T^{n}\varphi\}_{n=0}^{\infty} $ is a frame and $ T $ is normal, \textcolor{cyan}{Lemma $ \ref{NOR} $} leads us to get $ \| T\|=1 $, and then $ \| TT^{\ast}\|= 1 $. On the other hand, we have  
	$\| (TT^{\ast})^n f\|=\| T^n (T^{\ast})^n f\|\leq\| T\|^n \| (T^{\ast})^n f\|=\| (T^{\ast})^n f\|\rightarrow 0,$ for all $ f\in \h $ as $ n\rightarrow\infty$.
	Therefore, the result follows from [\cite{A3}, \textcolor{cyan}{Theorem 8}].
\end{proof}
\begin{rem}
	Consider a linearly independent frame sequence $ \{f_{k}\}_{k\in \mathbb{Z}}$ in a Hilbert space $ \h $
	which spans an infinite dimensional subspace. By using [\cite{A7}, \textcolor{cyan}{Proposition 2.1}] and [\cite{A6}, \textcolor{cyan}{Proposition 2.3}], there exists a linear invertible operator $ T: {\rm span}\{f_{k}\}_{k\in \mathbb{Z}}\rightarrow  {\rm span}\{f_{k}\}_{k\in \mathbb{Z}} $ such that $ Tf_k= f_{k+1} $. However, if $ \{f_{k}\}_{k\in \mathbb{Z}}$ is a frame sequence and the operator $ T $ is bounded, it has a unique extension to a bounded operator $ \widetilde{T}:\overline{\rm span}\{f_{k}\}_{k\in \mathbb{Z}}\rightarrow\overline{\rm span}\{f_{k}\}_{k\in \mathbb{Z}} $ such that
	\begin{equation*}
	\widetilde{T}\Big(\sum_{k\in \mathbb{Z}}c_k f_k\Big)=\sum_{k\in \mathbb{Z}}c_k f_{k+1}, \quad \{c_k\}_{k\in \mathbb{Z}}\in\ell^{2}(\mathbb{Z}).
	\end{equation*}
\end{rem}
By using previous remark and operator representation of dual frames, we can construct a frame in terms of its frame operator:
\begin{prop}\label{SF0}
	Let $ \{f_{k}\}_{k\in \mathbb{Z}}= \{T^{k}f_{0}\}_{k\in \mathbb{Z}} $ be a frame for $\h$ for some bounded, invertible and self-adjoint operator $ T:\h\to\h $ with the frame operator $ S $. Assume that $ V\in B(\h)$  and 
	$\{V^{k}f_m\}_{k\in \mathbb{Z}} $ is a dual frame of $ \{f_{k}\}_{k\in \mathbb{Z}}$  for some $m\in\Bbb{Z}$. Then $ \{S^{k}f_{0}\}_{k\in \mathbb{Z}} $ is a frame for $ \h $, whenever $ T $ is an isometry. 
\end{prop}
\begin{proof}
	We let $V^{k}f_m=g_k$ for all $k\in\Bbb{Z}$. It is clear that $Tf_k=f_{k+1}=T^{k+1}f_0$ for all $k\in\Bbb{Z}$ and $Tf= \sum_{k\in\mathbb{Z}}\langle f, g_k\rangle f_{k+1}$ for all $f\in\h$.
	On the other hand, by [\cite{A7}, \textcolor{cyan}{Lemma 3.3}],  $ V= (T^{\ast})^{-1} $. Since $ T $ is self-adjoint, we have
	$Tf =\sum_{k\in \mathbb{Z}} \langle f, T^{-k}f_m\rangle T^{k+1}f_{0},$ for all $f\in\h$.
	If $ T $ is an isometry, i.e., $ T^{\ast}T=I $,  then $ T=T^{-1} $, and therefore 
	we get
	\begin{equation*}
	Tf=\sum_{k\in \mathbb{Z}} \langle f, T^{k+m}f_{0}\rangle T^{k+1}f_{0}=T^{m+1}\sum_{k\in \mathbb{Z}} \langle f, T^{k}f_{0}\rangle T^{k}f_{0}=T^{m+1}Sf,
	\end{equation*}
	for all $f\in\h$.
	Hence, $T^m=S$, and we infer  that $ \{S^{k}f_{0}\}_{k\in \mathbb{Z}} $ is a frame for $ {\h} $.
\end{proof}
    It can be an interesting question whether the converse of \textcolor{cyan}{Proposition $ \ref{SF0} $} holds.
	We know that if $ \{S^{k}f_{0}\}_{k\in \mathbb{Z}} $ is a tight frame for $ \h $, [\cite{A7}, \textcolor{cyan}{Corollary 2.7}] shows that the frame operator $ S $ is an isometry. It is still an open question  whether $ T $ is an isometry or not.
\par
 Suppose that $ T $ is a bounded bijective operator on $\h$, and $ f_0\in\h $ such that $ \{T^n f_0\}_{n\in\mathbb{Z}} $ is a frame for $ \h $. We get that $ TST^{\ast}=S $, where $ S $ is the frame operator for $ \{T^n f_0\}_{n\in\mathbb{Z}} $. Indeed, 
\begin{equation*}
TST^{\ast}f=\sum_{n\in\mathbb{Z}}\langle T^{\ast}f, T^n f_0\rangle T^{n+1}f_0=\sum_{n\in\mathbb{Z}}\langle f, T^{n+1}f_0\rangle T^{n+1}f_0=Sf
\end{equation*}
In particular, $T$ is similar to a unitary operator. 
   \begin{prop}
	Let $ T\in GL(\h) $  and $ \varphi\in\h $ such that $ \{T^n \varphi\}_{n\in\mathbb{Z}} $ is a frame for $ \h $ with frame bounds $ A, B $ and frame operator $ S $. Let  $ U:=S^{-{1/2}}TS^{1/2} $ and $ \psi=S^{-1/2} \varphi $. Then $ \{U^n \psi\}_{n\in\mathbb{Z}} $ is a frame for $\h$ with bounds $ AB^{-1} $ and $ BA^{-1}.$ 
	\end{prop}
\begin{proof}
It is clear that $TST^{\ast}=S$ and $U$ is unitary (see [\cite{A116}, \textcolor{cyan}{Lemma 4.4}]). Since $U^n=S^{-{1/2}}T^nS^{1/2}$ for all $n\in\Bbb{Z}$, we have $\sum_{n\in\mathbb{Z}}|\langle f, U^n \psi\rangle|^2=
\sum_{n\in\mathbb{Z}}|\langle S^{-1/2}f, T^n \varphi\rangle|^2$. Then
\begin{equation*}
	\dfrac{A}{B}\|f\|^2\leq A\|S^{-1/2} f\|^2\leq\sum_{n\in\mathbb{Z}}|\langle f, U^n \psi\rangle|^2\leq B\|S^{-1/2}f\|^2\leq \dfrac{B}{A}\|f\|^2,\quad f\in\h.
	\end{equation*}
\end{proof}
As a minor modification in [\cite{A116}, \textcolor{cyan}{Corollary 4.5}], we also obtain the following result:
   \begin{prop}
	Let $ T\in GL(\h) $  and $ \varphi\in\h $ such that $ \{T^n \varphi\}_{n\in\mathbb{Z}} $ is a frame for $ \h $ with frame bounds $ A, B $.  Then 
\[\sqrt{\dfrac{A}{B}}\|f\|\leq\|T^nf\|\leq\sqrt{\dfrac{B}{A}}\|f\|,\quad \sqrt{\dfrac{A}{B}}\|f\|\leq\|(T^{\ast})^nf\|\leq\sqrt{\dfrac{B}{A}}\|f\|,\quad n\in\Bbb{Z},~f\in\h.\]
In particular, if $ \{T^n \varphi\}_{n\in\mathbb{Z}} $ is a tight frame, then $T^n$ and $(T^{\ast})^n$ are isometric for all $n\in\Bbb{Z}$.
	\end{prop}
\begin{proof}
Let $S$ denote the frame operator of
$ \{T^n \varphi\}_{n\in\mathbb{Z}}$ and let $ U:=S^{-{1/2}}TS^{1/2} $. Since $T$ is invertible, we infer that
 $U$ is unitary. Hence, for $f\in\h$ and $n\in\Bbb{Z}$ we have 
\[\dfrac{1}{\sqrt{B}}\|f\|\leq\|U^nS^{-{1/2}}f\|\leq\dfrac{1}{\sqrt{A}}\|f\|.\]
Therefore 
\[\sqrt{\dfrac{A}{B}}\|f\|\leq\|S^{1/2}U^nS^{-{1/2}}f\|=\|T^n f\|=\|S^{1/2}U^nS^{-{1/2}}f\|\leq\sqrt{\dfrac{B}{A}}\|f\|.\]A similar calculation applies to $\|(T^{\ast})^nf\|.$
\end{proof}
\par
Let $ T\in GL(\h) $. Similarly as in \cite{A116}, we define the set
\begin{center}
	$\V_{\Bbb{Z}}(T)$:=\Big\{$f\in\h:~\{T^n f\}_{n\in\mathbb{Z}}$ is a frame for $ \h $\Big\}.
\end{center}
\textcolor{cyan}{Proposition 4.11} of \cite{A116} shows that from one vector $\varphi\in\V_{\Bbb{Z}}(T)$ (if it exists) we obtain all vectors in $\V_{\Bbb{Z}}(T)$. Indeed, 
$\V_{\Bbb{Z}}(T)=\Big\{V\varphi:~V \in GL(\h)~\mbox{and}~  VT=TV\Big\}.$
\begin{prop}
 Assume that $T\in GL(\h)$, $\varphi\in\V_{\Bbb{Z}}(T) $ and $ V $ is a unitary operator such that $ VT=TV $.  
 Let  $S$ and $ \widetilde{S} $ be the frame operators for
 $\{T^n \varphi\}_{n\in\mathbb{Z}}$ and $\{T^n V\varphi\}_{n\in\mathbb{Z}}$, respectively. Then $ \{(\widetilde{S})^n f\}_{n\in\mathbb{Z}} $ is a frame for $ \h $ if and only if $ \{S^n V^{\ast}f\} $ is a frame for $ \h $. In other words,
 $f\in\V_{\Bbb{Z}}(\widetilde{S})$ if and only if $V^{\ast}f\in\V_{\Bbb{Z}}(S).$
\end{prop}
\begin{proof}
For each $f\in\h$, we have
 \begin{equation*}
 \begin{aligned}
 \widetilde{S}f=\sum_{n\in\mathbb{Z}}\langle f, T^n V\varphi\rangle T^n V\varphi=\sum_{n\in\mathbb{Z}}\langle f,V T^n \varphi\rangle VT^n \varphi
 =V\sum_{n\in\mathbb{Z}}\langle V^{\ast}f, T^n \varphi\rangle T^n\varphi= VSV^{\ast}f.
 \end{aligned}
 \end{equation*}
 As $ V $ is unitary, we get $ (\widetilde{S})^n=VS^n V^{\ast} $ and  $V^{\ast} (\widetilde{S})^n=S^n V^{\ast}$ which immediately yields the desired conclusion.
\end{proof}
\section{Frame representation of the form  $ \{a_n T^n \varphi\}_{n=0}^{\infty} $}
In this section, we generalize some results in the recent papers \cite{A6, A117} which have been proved by Christensen et al.  We consider frames of the form
$\{f_k\}_{k=1}^{\infty}=\{a_n T^nf_1\}_{n=0}^{\infty}$ for some scalars $ a_n\neq 0 $ with $ \sup_n\Big|\dfrac{a_n}{a_{n+1}}\Big|<\infty $ and a bounded linear operator $ T:  {\rm span}\{f_k\}_{k=1}^{\infty}\rightarrow\h $. Using \cite{A117}, we define $ \T_{\omega}:\ell^{2}({\mathbb{N}}_{0})\rightarrow\ell^{2}({\mathbb{N}}_{0}) $ by $\T_{\omega}\{c_i\}_{i=0}^{\infty} =\Big(0,\frac{a_0}{a_1}c_0, \frac{a_1}{a_2}c_1,\cdots\Big).$
The following theorem was proved in \cite{A117}:
\begin{thm}\label{EM2}
	Let  $ \{a_n\}_{n=0}^{\infty} $ be  a sequence of non-zero scalars  with $ \sup_n \Big|\dfrac{a_n}{a_{n+1}}\Big|<\infty $, and let $ \{f_k\}_{k=1}^{\infty}=\{a_n T^n f_1\}_{n=0}^{\infty} $   be a linearly independent frame for an infinite-dimensional Hilbert space $ \h $, where  $ T:  {\rm span}\{f_k\}_{k=1}^{\infty}\rightarrow\h $ is a linear operator. Then $T$ is bounded if and only if $ {\N}_U $ is invariant under $ \T_{\omega} $.
\end{thm}
The condition $\sup_n \Big|\dfrac{a_n}{a_{n+1}}\Big|<\infty$ is indeed necessary for frames of the form $ \{ a_n T^n \varphi\}_{n=0}^{\infty}$ when $T\in B(\h) $.
\begin{prop}
	Assume that $ T\in B(\h) $  such that $ \{ a_n T^n \varphi\}_{n=0}^{\infty} $ is a frame for some $ \varphi\in\h $ and some non-zero scalars $ \{a_n\}_{n=0}^{\infty} $.  Then $\sup_n \Big|\dfrac{a_n}{a_{n+1}}\Big|<\infty.$
\end{prop}
\begin{proof}
Let $A$ and $B$ be frames bounds of $ \{f_k\}_{k=1}^{\infty}=\{ a_n T^n \varphi\}_{n=0}^{\infty}.$  Using that $\sqrt{A}\leq\|f_k\|\leq\sqrt{B}$ for all $k\in\Bbb{N}$, we get
\[\|f_k\|\|T\|\geq\|Tf_k\|=\Big\|\dfrac{a_{k-1}}{a_k}f_{k+1}\Big\|\geq\Big|\dfrac{a_{k-1}}{a_k}\Big|\sqrt{A}\geq\Big|\dfrac{a_{k-1}}{a_k}\Big|\sqrt{\dfrac{A}{B}}\|f_k\|.\] Then  $\sup_n \Big|\dfrac{a_n}{a_{n+1}}\Big|\leq \sqrt{\dfrac{B}{A}}\|T\|$ as desired.
\end{proof}
\par
If  $ T:\h\rightarrow\h $ is a  linear operator and $\{f_k\}_{k=1}^\infty= \{ a_n T^n \varphi\}_{n=0}^{\infty} $ is a frame (with frame bounds $A$ and $B$) for some $\varphi\in\h $ and  some non-zero scalars $ \{a_n\}_{n=0}^{\infty} $ with $\sup_n \Big|\dfrac{a_n}{a_{n+1}}\Big|<\infty$, then we have
\[\|Tf_k\|=\Big\|\dfrac{a_{k-1}}{a_k}f_{k+1}\Big\|\leq\Big|\dfrac{a_{k-1}}{a_k}\Big|\sqrt{B}\leq\sqrt{\dfrac{B}{A}}\|f_k\|,\quad k\in\Bbb{N}.\] 
In this case $T$ may be unbouded (see \textcolor{cyan} {Proposition} \ref{Tunbounded}).
Using [\cite{A6}, \textcolor{cyan}{Proposition 2.5}], we can obtain the following result  for a frame in the form $ \{a_n T^n \varphi\}_{n=0}^{\infty} $.
\begin{prop}
	Assume that $ T\in B(\h) $ such that $\{ a_n T^n \varphi\}_{n=0}^{\infty} $ is a frame for some $ \varphi\in\h $ and some non-zero scalars $ \{a_n\}_{n=0}^{\infty}$.
	Then  $ T $ has closed range and $ \R_T=\overline{\rm span}\{a_n T^{n+1} \varphi\}_{n=0}^{\infty} $.
\end{prop}
\begin{proof}
	Using [\cite{A4}, \textcolor{cyan}{Theorem 5.5.1}], the synthesis operator 
	\begin{equation*}
	U: \ell^{2}(\mathbb{N}_0)\rightarrow \h ,\quad U(c_0, c_1, c_2,...)=\sum_{i=0}^{\infty}c_i a_i T^i \varphi
	\end{equation*}
	is surjective. Letting $ x\in\h $ there exists $ (c_0, c_1, c_2,...)\in\ell^{2}(\mathbb{N}_0) $ such that $ x=\sum_{i=0}^{\infty}c_i a_i T^i \varphi $. Therefore
	\begin{equation*}
	Tx=\sum_{i=0}^{\infty}c_i a_i T^{i+1} \varphi\in\overline{\rm span}\{a_i T^{i+1} \varphi\}_{i=0}^{\infty}.
	\end{equation*}
	Therefore $ \R_T \subseteq \K:=\overline{\rm span}\{a_i T^{i+1} \varphi\}_{i=0}^{\infty}$. On the other hand, $ \{a_i T^{i+1} \varphi\}_{i=0}^{\infty} $ is a frame for $ \K $, and then its synthesis operator is surjective. Letting $ x\in \K $, there is $ (c_0, c_1, c_2,...)\in\ell^2(\mathbb{N}_0) $ such that $ x=\sum_{i=0}^{\infty}c_i a_i T^{i+1} \varphi=T \sum_{i=0}^{\infty}c_i a_i T^{i} \varphi\in \R_T$. Therefore $ \R_T=\overline{\rm span}\{a_n T^{n+1} \varphi\}_{n=0}^{\infty} $, i.e., $ T $ has closed range.
\end{proof}
The following proposition generalize a result in \cite{A13, A113}, where we characterize the availability of
the representation $ \{f_k\}_{k=1}^{\infty}=\{a_n T^n f_1\}_{n=0}^{\infty}$.
\begin{prop}
Let $ \{f_k\}_{k=1}^{\infty} $ and $ \{g_k\}_{k=1}^{\infty} $ be sequences in $\h$
 such that each $ f\in \h $ has the convergent expansion
	\begin{equation}\label{Q1}
	f=\sum_{k=1}^{\infty}\langle f, g_k\rangle f_k.
	\end{equation}
	Suppose that $\{a_n\}_{n=0}^\infty$ is a sequence of non-zero scalars such that for any $ f\in\h $ the series $ \sum_{k=1}^{\infty}\langle f, g_k\rangle \dfrac{a_{k-1}}{a_{k}}f_{k+1} $ converges. Then $ \{f_k\}_{k=1}^{\infty}=\{a_n T^n f_1\}_{n=0}^{\infty} $ for some $ T\in B(\h) $  if and only if 
	\begin{equation}\label{D1}
	f_{j+1}=\dfrac{a_j}{a_{j-1}}\sum_{k=1}^{\infty}\langle f_j, g_k \rangle \dfrac{a_{k-1}}{a_k}f_{k+1},\quad j\in\mathbb{N}.
	\end{equation}
\end{prop}	
\begin{proof}
	Assume that $ \{f_k\}_{k=1}^{\infty} $ can be represented as $ \{a_n T^n f_1\}_{n=0}^{\infty} $ for some $ T\in B(\h)$.  Then $Tf_k=\dfrac{a_{k-1}}{a_{k}}f_{k+1}$ for all $k\in\Bbb{N}$. By applying $ T $ on $ (\ref{Q1}) $, we get
	\begin{equation*}
	Tf=\sum_{k=1}^{\infty}\langle f, g_k\rangle Tf_k=\sum_{k=1}^{\infty}\langle f, g_k\rangle\dfrac{a_{k-1}}{a_{k}}f_{k+1},\quad f\in\h. 
	\end{equation*}
	Letting $ f=f_j$ in the above expression, it follows that $ \dfrac{a_{j-1}}{a_j}f_{j+1}=\sum_{k=1}^{\infty}\langle f_j, g_k\rangle\dfrac{a_{k-1}}{a_{k}}f_{k+1} $, and we get (\ref{D1}).
	\par
	For the opposite implication, suppose that $ (\ref{D1}) $ holds.  Define the linear operator
	\begin{equation*}
	T:\h\to\h,\quad Tf=\sum_{k=1}^{\infty}\langle f, g_k\rangle\dfrac{a_{k-1}}{a_{k}}f_{k+1},\quad f\in\h.
	\end{equation*}
	By uniform boundedness principle,  $T$  is bounded.  Then by $ (\ref{D1}) $ we conclude that 
	$Tf_j=\sum_{k=1}^{\infty}\langle f_j, g_k\rangle\dfrac{a_{k-1}}{a_{k}}f_{k+1}=\dfrac{a_{j-1}}{a_j}f_{j+1}$ for all $j\in\mathbb{N}.$
	Therefore $ \{f_k\}_{k=1}^{\infty}=\{a_n T^n f_1\}_{n=0}^{\infty} $.
\end{proof}
Motivated by \textcolor{cyan}{Proposition 2.6 } in \cite{A6} and with a small change in its proof, we can obtain the following result which generalizes it.
\begin{prop}\label{Tunbounded}
	Assume that the frame $ \{f_k\}_{k=1}^{\infty} $ is linearly independent, contains a Riesz basis and has finite and strictly positive excess.  Let $ T:\h\rightarrow\h $ be a linear operator  such that $\{f_k\}_{k=1}^{\infty}=\{a_n T^n f_1\}_{n=0}^{\infty}$
for some non-zero scalars $ \{a_n\}_{n=0}^{\infty} $ with $ \sup_n\Big|\dfrac{a_n}{a_{n+1}}\Big|<\infty $ and $ \inf_n\Big|\dfrac{a_n}{a_{n+1}}\Big|>0 $. Then $ T $ is unbounded.
\end{prop}
\begin{proof}
Let $\delta:= \inf_n\Big|\dfrac{a_n}{a_{n+1}}\Big|$ and $\gamma:= \sup_n\Big|\dfrac{a_n}{a_{n+1}}\Big|.$ By assumption there exists $m\in\Bbb{N}$ such that $ \{f_k\}_{k=m+1}^{\infty} $ is a Riesz basis for $\K:=\overline{\rm span}\{f_k\}_{k=m+1}^{\infty}$ and 
$\{f_k\}_{k=m}^{\infty}$ is an overcomplete frame for $\K$. Since $0<\delta\leq\gamma<\infty$, we infer  that $ \Big\{\dfrac{a_{k-1}}{a_k}f_{k+1}\Big\}_{k=m}^{\infty} $ is  a Riesz basis for $\K$, and we denote  its lower Riesz basis bound by $A$. For each $n\in\Bbb{N}$, let $A_n$ denote the optimal lower Riesz basis bound for the finite sequence  $ \{f_k\}_{k=m}^{m+n-1} $. Since  $ \{f_k\}_{k=m}^{\infty} $ is  a  linearly independentan and overcomplete frame, it follows $A_n\to 0$ as $n\to\infty$ by \textcolor{cyan}{Proposition 7.2.1} in \cite{A4}. Let  $n\in\Bbb{N}$, then there exists a non-zero sequence $\{c_k\}_{k=m}^{m+n-1}$ such that 
\[\Big\|\sum_{k=m}^{m+n-1}c_kf_k\Big\|^2\leq (A_n+\dfrac{1}{n})\sum_{k=m}^{m+n-1}|c_k|^2.\] Then
\[\begin{aligned}
\Big\|T\sum_{k=m}^{m+n-1}c_kf_k\Big\|^2&=\Big\|\sum_{k=m}^{m+n-1}c_k\dfrac{a_{k-1}}{a_k} f_{k+1}\Big\|^2\\
&\geq A\sum_{k=m}^{m+n-1}|c_k|^2\\
&\geq \dfrac{A}{A_n+\dfrac{1}{n}} \Big\|\sum_{k=m}^{m+n-1}c_kf_k\Big\|^2.
\end{aligned}\]
If $T$ is bounded, then it follows from the above inequlity that $\|T\|\geq  \dfrac{A}{A_n+\dfrac{1}{n}}$. Since  $\dfrac{A}{A_n+\dfrac{1}{n}}\to\infty$ as $n\to\infty$, which is a contradiction.
\end{proof}
	\section{ Some auxiliary results: perturbation of a frame $ \{T^n \varphi\}_{n=0}^{\infty} $}

Motivated by some results about perturbations of  frames of the form $ \{T^n \varphi\}_{n=0}^{\infty} $ in \cite{A13}, 
we give some results by restricting ourself to perturb a frame
$ \{T^n \varphi\}_{n=0}^{\infty} $ with elements from a subspace on which $T$ acts as a contraction.
We also state some stability results obtained by considering perturbations of  operators  belonging to an invariant subspace.
 \begin{prop}
Assume that $\{T^n\varphi\}_{n=0}^\infty$ is a Riesz sequence for some $T\in B(\h)$ and some $\varphi\in\h$, and let $A$ denote a lower Riesz bound.
Suppose that $V\subset\h$ is invariant under $T$ and that there exists $\mu \in [0,1)$ such
that $\|T\psi\|\leq\mu\|\psi\|.$ Then 
 $\{T^n(\varphi+\psi)\}_{n=0}^\infty$ is a Riesz sequence  for all $\psi\in V$ for which $\|\psi\|<(1-\mu)\sqrt{A}.$
	\end{prop}
\begin{proof}
It is clear that $\sum_{n=0}^\infty\|T^n\psi\|^2<\infty$ for all $\psi\in V$. By  [\cite{A8}, \textcolor{cyan}{Theorem 2.14}] it is sufficient to show that  $\sum_{n=0}^\infty\|T^n(\varphi+\psi)-T^n\varphi\|\|S^{-1}T^n\varphi\|<1,$ where $S$ is frame operator for $\{T^n\varphi\}_{n=0}^\infty$. Since $\|S^{-1}T^n\varphi\|\leq 1/\sqrt{A},$ we have
\[\sum_{n=0}^\infty\|T^n(\varphi+\psi)-T^n\varphi\|\|S^{-1}T^n\varphi\|\leq\dfrac{\|\psi\|}{\sqrt{A}}\sum_{n=0}^\infty\mu^n=\dfrac{\|\psi\|}{(1-\mu)\sqrt{A}}<1,\]
as desired.
\end{proof}
A similar approach as in the proof of \textcolor{cyan}{Proposition 3.3} in \cite{A13}
yields the following result.
 \begin{prop}
	Let $\{a_n\}_{n=0}^\infty$ be a bounded sequence of scalars. Assume that $\{a_nT^n\varphi\}_{n=0}^\infty$ is a frame for some bounded linear
operator $T:\h\to\h$ and some $\varphi\in\h$, and let $A$ denote a lower frame bound.
Suppose that $V\subset\h$ is invariant under $T$ and that there exists $\mu \in [0,1)$ such
that $\|T\psi\|\leq\mu\|\psi\|.$ Then the following hold:
\begin{enumerate}
\item[$(i)$] $\{a_nT^n(\varphi+\psi)\}_{n=0}^\infty$ is a frame sequence for all $\psi\in V$.
\item[$(ii)$] 
 $\{a_nT^n(\varphi+\psi)\}_{n=0}^\infty$ is a frame  for all $\psi\in V$ for which $\sup_n\|a_n\psi\|<\sqrt{A(1-\mu^2)}.$
\end{enumerate}
	\end{prop}
We now provide a perturbation result which can be used to construct a frame with representation $ \{a_n T^n \varphi\}_{n=0}^{\infty} $.
\begin{prop}
	Let $ T\in B(\h)$  and $\varphi, \psi\in\h$. Assume that $ \{a_n\}_{n=0}^{\infty} $ is sequence of non-zero scalars such that $ \{a_n T^n\varphi\}_{n=0}^{\infty} $ is a frame for $ \h $ with lower bound $ A $ and $ \{a_{n+1} T^n \psi\}_{n=0}^{\infty} $ is a Bessel sequence for $ \h $ with Bessel bound $ B $. If  $\sup_n\Big|\dfrac{a_n}{a_{n+1}}\Big|<\sqrt{\dfrac{A}{B}} $, then $ \{a_n T^n (\varphi+\psi)\}_{n=0}^{\infty} $ is a frame for $ \h $.
\end{prop}
\begin{proof} 
Let $ \{c_n\}_{n=0}^{\infty}\in\ell^{2}(\mathbb{N}_0)$ and $\alpha:= \sup_n\Big|\dfrac{a_n}{a_{n+1}}\Big|$. By assumption, we have 
	\begin{equation*}
	\begin{aligned}
	\Big\|\sum_{n=0}^{\infty}c_n(a_n T^n \varphi- a_n T^n (\varphi+\psi))\Big\|^2&=\Big\|\sum_{n=0}^{\infty}c_na_n T^n \psi\Big\|^2\\
	&=\sup_{\|f\|=1}\Big|\Big\langle\sum_{n=0}^{\infty}c_na_n T^n \psi, f\Big\rangle\Big|^2\\
	&=\sup_{\|f\|=1}\Big|\sum_{n=0}^{\infty}c_n\dfrac{a_n}{a_{n+1}}\langle a_{n+1}T^n \psi, f\rangle\Big|^2\\
	&\leq\sum_{n=0}^{\infty}\Big|c_n\dfrac{a_n}{a_{n+1}}\Big|^2\sup_{\|f\|=1} \sum_{n=0}^{\infty}\big|\langle a_{n+1}T^n \psi, f\rangle\big|^2\\
	&\leq\alpha^2B\sum_{n=0}^{\infty}|c_n|^2.
	\end{aligned}
	\end{equation*}
	Hence, [\cite{A4}, \textcolor{cyan}{Theorem 22.1.1}] implies that the desired result.
\end{proof}
Here $\B$ denotes the set of  bounded linear operators  $T: \h\to\h $ for which there exist $\lambda_T\in [0, 1)$ and an invariant subspace
$V_{T}\subset\h$ under $T$ such that $ \| T\varphi\|\leq\lambda_T\| \varphi\|$ for all $\varphi\in V_T$.
In the following proposition $I$ is  a countable index set and $ \{g_j\}_{j\in I}$ is a sequence in $\h $.
    \begin{prop}
	Suppose that $ T, W\in\B $ and $ \{g_j\}_{j\in I}\subseteq V_W \cap V_T $. Let $ \{W^n g_j\}_{n\geq 0, j\in I} $ be a Riesz sequence with frame operator $ S $, and  $ \{T^n g_j\}_{n\geq 0, j\in I} $ be a Bessel sequence for $ \h $.  Assume that  $\sum_{j\in I}\|g_j\|^{2}<\dfrac{1-\lambda^{2}}{2\|S^{-1}\|} $, where $ \lambda:= \max\{\lambda_W, \lambda_T\} $. Then $ \{T^n g_j\}_{n\geq 0, j\in I} $ is a Riesz sequence.
	\end{prop}
\begin{proof}
By assumptions, we have 
\begin{equation*}
 \|W g_j\|\leq \lambda\|g_j\|,\quad  \|T g_j\|\leq \lambda\|g_j\|,\quad  j\in I.
\end{equation*}
Then
\begin{equation*}
\begin{aligned}
\sum_{j\in I}\sum_{n=0}^{\infty}\|W^n g_{j}- T^n g_{j}\|\| S^{-1}W^n g_{j}\|
&\leq \sum_{j\in I}\sum_{n=0}^{\infty}\|W^n g_{j}- T^n g_{j}\|\|S^{-1}\|\| W^n g_j\|\\
&\leq \sum_{j\in I}\sum_{n=0}^{\infty}(\|W^n g_j\|+\|T^n g_j\|)\|S^{-1}\|\|W^n g_j\|
\\
&\leq2\|S^{-1}\|\sum_{j\in I}\sum_{n=0}^{\infty}\lambda^{2n}\|g_j\|^{2}\\
&=\dfrac{2\|S^{-1}\|}{1- \lambda^{2}}\sum_{j\in I}\|g_j\|^{2}<1.
\end{aligned}
\end{equation*}
 Therefore, [\cite{A8}, \textcolor{cyan}{Theorem 2.14}] leads to the desired result.
\end{proof}
\begin{prop}
	Let $ T, W\in \B $ and $ \varphi\in V_T \cap V_W $. Suppose  that $ \{T^{n}\varphi\}_{n=0}^{\infty} $ is a frame for $ \h $ with lower frame bound $ A $ and $ \{W^{n}\varphi\}_{n=0}^{\infty} $ is a Bessel sequence for $ \h $. Let $2\| \varphi\|< \sqrt{A(1-\lambda^{2}} $, where $ \lambda:=\max\{\lambda_T, \lambda_W\} $. Then $ \{W^{n}\varphi\}_{n=0}^{\infty} $ is a frame for $ \h $.
\par
	In the case where $ \{T^{n}\varphi\}_{n=0}^{\infty} $ is a Riesz sequence  with lower bound $ A $, then $ \{T^{n}\varphi+W^{n}\varphi\}_{n=0}^{\infty} $ is a Riesz sequence, whenever $\| \varphi\|< \sqrt{A(1-\lambda^{2})}.$
\end{prop}
\begin{proof}
	By assumptions, we have
	\begin{equation*}
	\begin{aligned}
	\sum_{n=0}^{\infty}\| T^{n}\varphi - W^{n}\varphi\|^{2}
	\leq 2\Big(\sum_{n=0}^{\infty}\| T^{n}\varphi\|^{2}+ \sum_{n=0}^{\infty}\| W^{n}\varphi\|^{2}\Big)
	\leq 4\| \varphi\|^{2}\sum_{n=0}^{\infty}\lambda^{2n}
	=\dfrac{4\| \varphi\|^{2}}{1-\lambda^{2}}< A.
	\end{aligned}
	\end{equation*}
	We conclude by  [ \cite{A4}, \textcolor{cyan}{Corollary $22.1.5$}]  that   $ \{W^{n}\varphi\}_{n=0}^{\infty} $ is a frame for $ \h .$ 
	\par 
	If $ \{T^{n}\varphi\}_{n=0}^{\infty} $ be a Riesz sequence, then
	\begin{equation*}
	\begin{aligned}
	\Big\| \sum_{n=0}^{\infty}c_{n}(T^{n}\varphi - (T^{n}\varphi+W^{n}\varphi) )\Big\|^2 &=\Big\| \sum_{n=0}^{\infty}c_{n}W^{n}\varphi\Big\|^2\\
	&\leq \sum_{n=0}^{\infty}|c_{n}|^{2}\sum_{n=0}^{\infty}\| W^{n}\varphi\|^{2}\\
	&\leq\dfrac{\| \varphi\|^2}{1- \lambda^{2}}\sum_{n=0}^{\infty}|c_{n}|^{2}.
	\end{aligned}
	\end{equation*}
	Therefore, the result follows from [\cite{A4}, \textcolor{cyan}{Theorem $ 22.3.2 $}].
\end{proof}
\subsection*{Acknowledgment} The authors would like to thank the anonymous reviewers whose comments helped us improve the presentation of the paper.

\end{document}